\documentclass[12pt]{article}
\usepackage[german,english]{babel}
\usepackage[latin1]{inputenc}
\usepackage{times,amsmath,amssymb}
\usepackage{amsfonts,graphicx}
\newtheorem{thm}{Theorem}

\newcommand{\DD}{{\mathbb D}}
\newcommand{\LL}{{\mathbb L}}
\newcommand{\PP}{{\mathbb P}}
\newcommand{\RR}{{\mathbb R}}
\newcommand{\cC}{{\mathcal C}}
\newcommand{\cD}{{\mathcal D}}
\newcommand{\cL}{{\mathcal L}}
\newcommand{\cS}{{\mathcal S}}

\newcommand{\dis}
                 {{\mathrel{\scriptstyle{\triangle}}}}
\newcommand{\eps}{{\varepsilon}}
\newcommand{\GL}{{\mathrm{GL}}}
\newcommand{\welle}[1]{{\widetilde{#1}}}
\let\phi=\varphi

\let\theta=\vartheta

\newcommand{\DelimArray}[4]{\left#1\begin{array}{*{#3}{c}}#4\end{array}\right#2}

\newcommand{\Mat}{\DelimArray()}

\newenvironment{proof}
    {\begin{trivlist} \item {\emph{Proof.}}} 
    {{}\hfill $\square$ \end{trivlist}}
\newcounter{abbildung} 
\begin{document}
\selectlanguage{english}

\centerline{\Large\bf A Three-Dimensional Laguerre Geometry}
\centerline{\Large\bf and Its Visualization} \vspace*{1.1cm}

 \centerline{\large Hans Havlicek and Klaus List, Institut f{\"u}r Geometrie, TU Wien}

 \vspace*{0.8cm}

{\small We describe and visualize the chains of the
$3$-dimensional chain geometry over the ring $\RR(\eps)$,
$\eps^3=0$.
 }
%
%
%
%
%
%
%
%

\section{Introduction}\label{sect:intro}

The aim of the present paper is to discuss in some detail the
\emph{Laguerre geometry}\/ (cf.\ \cite{benz-73}, \cite{herz-95})
which arises from the $3$-dimensional real algebra
$\LL:=\RR(\eps)$, where $\eps^3=0$. This algebra generalizes the
algebra of real dual numbers $\DD=\RR(\eps)$, where $\eps^2=0$.
The Laguerre geometry over $\DD$ is the geometry on the so-called
\emph{Blaschke cylinder}\/ (Figure \ref{abb:1}); the
non-degenerate conics on this cylinder are called \emph{chains}\/
(or \emph{cycles, circles}\/). If one generator of the cylinder is
removed\linebreak
\begin{minipage}[t]{11.2cm}
then the remaining points of the cylinder are in one-one
correspondence  (via a \emph{stereographic projection}\/) with the
points of the plane of dual numbers, which is an isotropic plane;
the chains go over to circles and non-isotropic lines. So the
point space of the chain geometry over the real dual numbers can
be considered as an affine plane with an extra ``improper line''.
\par
The Laguerre geometry based on $\LL$ has as point set the
projective line $\PP(\LL)$ over $\LL$. It can be seen as the real
affine $3$-space on $\LL$ together with an ``improper affine
plane''. There is a point model for this geometry, like the
Blaschke cylinder, but it is more complicated, and belongs to a
$7$-dimensional projective space (\cite[p.~812]{herz-95}). We are
not going to use it. Instead, we describe $\PP(\LL)$ as an
extension of the affine space on $\LL$ by ``improper points''
which will be described via lines, parabolas, and cubic parabolas.
\end{minipage}\hfill{\unitlength1cm
\begin{minipage}[t]{4.8cm}
\vspace*{-0mm}
      \begin{center}
         \begin{picture}(3.91,6)
         \put(0.0 ,0.0){\includegraphics[height=6cm]{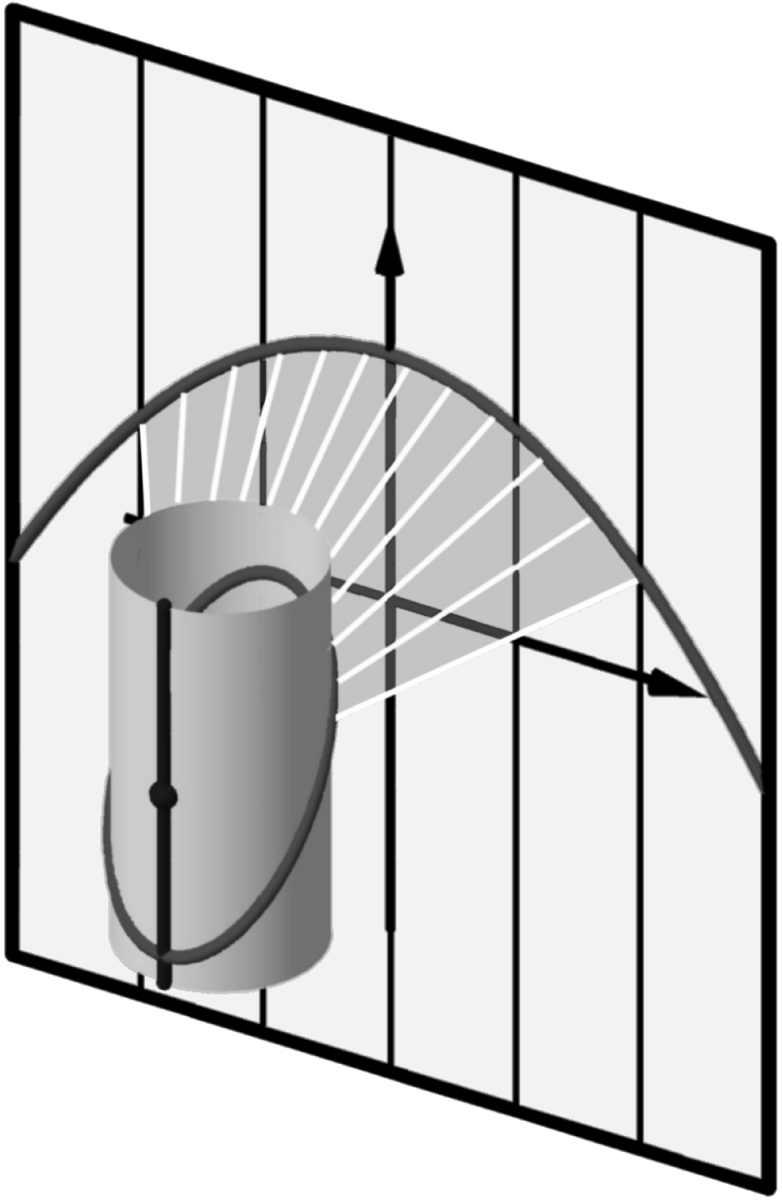}}
         \put(3.3,2.1){\small$\RR$}
         \put(2.05,4.8){\small$\RR\eps$}
         \end{picture}
         \refstepcounter{abbildung}\label{abb:1}%
         {\small Figure \ref{abb:1}}%
      \end{center}
      \end{minipage}
         }%

\section{Higher order contact of twisted cubics}\label{sect:kubiken}

Here we present some results which will be needed in Section
\ref{sect:laguerre}. We refer to \cite{bol-50}, \cite{brau-76-2},
and \cite{semp+k-98} for the basic properties of twisted cubics in
the real projective space $\PP_3(\RR)$.
\begin{thm}\label{thm:kubik}
Let $C$ and $\welle C$ be twisted cubics of\/ $\PP_3(\RR)$ with a
common point $f$, a common tangent line $F$ at $f$, and a common
osculating plane $\Phi$ at $f$. Then a collineation of\/
$\PP_3(\RR)$ taking $C$ to $\welle C$ is uniquely determined by
each of the following conditions:\\
\begin{tabular}{ll}
 \emph{(I)} & All lines of the pencil $\cL(f,\Phi)$ are invariant.\\
 \emph{(II)} & All points of the line $F$ are invariant.\\
 \emph{(III)} & All planes of the pencil with axis $F$ are invariant.
\end{tabular}
\end{thm}
\begin{proof}
(I) We recall that distinct tangent lines of a twisted cubic are
skew. The tangent surface of $C$ intersects the osculating plane
$\Phi$ in a curve which is the union of $F$ and a conic $K$
through $f$; the line $F$ is tangent to $K$. Likewise the tangent
surface of $\welle C$ yields a conic $\welle K$. Let $r_i$,
$i\in\{1,2\}$, be distinct points of $C\setminus\{f\}$. The
tangent lines of $C$ at these points meet the plane $\Phi$ at
points $k_i\in K\setminus\{f\}$, whence the lines $L_i:=f\vee k_i$
are distinct. These lines meet $\welle K$ residually at points
$\welle k_i$ which in turn are incident with tangent lines of
$\welle C$ at distinct points $\welle r_i$. These points $\welle
r_i$ are determined uniquely. So, every collineation of type (I)
takes $r_i$ to $\welle r_i$, and $f$ to $f$. Conversely, there is
a unique collineation $\kappa$ of $\PP_3(\RR)$ with
$C^\kappa=\welle C$, $r_i^\kappa=\welle r_i$, and $f^\kappa=f$.
Since $F$, $f\vee k_1$, and $f\vee k_2$ remain invariant under
$\kappa$, all lines of the pencil $\cL(f,\Phi)$ remain fixed. So
this $\kappa$ is the only collineation with the required
properties.
\par
(II) The proof runs in a similar manner. The osculating planes at
$r_i$ meet $F$ at points $k_i\neq f$. Now $\welle r_i\in \welle
C\setminus\{f\}$ can be chosen such that their osculating planes
meet the line $F$ at $k_i$.
\par
(III) Each of the planes $F\vee r_i$ meets the twisted cubic
$\welle C$ residually at a point $\welle r_i$. Now we can proceed
as above.
\end{proof}

Let $p_0$, $p_3$, and $p$ be three distinct points of $C$. Define
the point $p_1$ as the intersection of the tangent line at $p_0$
with the osculating plane at $p_3$. Likewise, by changing the role
of $p_0$ and $p_3$, a point $p_2$ is obtained. Then
$(p_0,p_1,p_2,p_3,p)$ is a frame of reference such that
\begin{equation}\label{eq:kubik-homogen}
  C = \{\RR(s^3,s^2t,st^2,t^3) \mid (0,0)\neq (s,t)\in\RR^2\}.
\end{equation}
We assume that $f=p_3=\RR(0,0,0,1)$, whence $F$ is given by
$x_0=x_1=0$ and $\Phi$ has an equation $x_0=0$. A collineation of
$\PP_3(\RR)$ is of type (I), (II) or (III) if, and only if, it has
a regular matrix with one of the following forms:
\begin{equation}\label{eq:matrizen}
\mbox{(I): }
  \Mat4{1 & a_{01} &a_{02} & a_{03}\\
        0 & a_{11} &0      & a_{13}\\
        0 & 0      &a_{11} & a_{23}\\
        0 & 0      &0      & a_{33}
    }
    ,\mbox{ (II): }
  \Mat4{1 & a_{01} &a_{02} & a_{03}\\
        0 & a_{11} &a_{12} & a_{13}\\
        0 & 0      &a_{22} & 0     \\
        0 & 0      &0      & a_{22}
    },\mbox{ (III): }
    \Mat4{1 & 0 &a_{02} & a_{03}\\
          0 & 1 &a_{12} & a_{13}\\
          0 & 0         &a_{22} & a_{23}\\
          0 & 0         &0      & a_{33}
    }
\end{equation}
Next we describe higher-order contact of twisted cubics; cf.\ also
\cite[pp.~211--219]{bol-50}.

\begin{thm}\label{thm:2}
Let $C$ be the twisted cubic (\ref{eq:kubik-homogen}) and let
$\kappa$ be a collineation of\/ $\PP_3(\RR)$ given by one of the
matrices (\ref{eq:matrizen}). Then the conditions stated in the
first row, in the first and the second row, and in all rows of the
table below are necessary and sufficient for the twisted cubics
$C$ and $C^\kappa$ to have contact at the point $f=\RR(0,0,0,1)$
of order\/ $2$, $3$, and\/ $4$, respectively.
\begin{equation}\label{eq:tabelle}
\begin{array}{|c|ll||ll||ll|}
  \hline
  &\multicolumn{2}{c||}{\mbox{\emph{(I)}}}
                      &\multicolumn{2}{c||}{\mbox{\emph{(II)}}}
                      &\multicolumn{2}{c|}{\mbox{\emph{(III)}}}\\
  \hline
   1&a_{33}=a_{11}&              &a_{22}=a_{11}&             & a_{33}=a_{22}^2& \\
   2&a_{11}=1,& a_{23}=-a_{01} &a_{11}=1,& a_{01}=2a_{12}& a_{22}=1,& a_{23}=2a_{12}\\
   3&a_{01}=0,& a_{13}=2a_{02} &a_{12}=0,& a_{13}=2a_{02}& a_{12}=0,& a_{13}=2a_{02}\\
  \hline
\end{array}
\end{equation}
\end{thm}

\begin{proof}
The quadratic forms {\arraycolsep0.5mm\begin{eqnarray*}
  G_1:\RR^4\to\RR   &:& (x_0,x_1,x_2,x_3)\mapsto x_0x_3-x_1x_2,
   \\
  G_2:\RR^4\to\RR   &:& (x_0,x_1,x_2,x_3)\mapsto x_1x_3-x_2^2,
\end{eqnarray*}}
define a hyperbolic quadric $x_0x_3-x_1x_2=0$ and a quadratic cone
$x_1x_3-x_2^2=0$ with vertex $p_0$. Their intersection is the
twisted cubic $C$ and the line $x_2=x_3=0$. The tangent planes of
the two surfaces at $f$ are different. Let $\kappa$ be given by a
matrix $A$ of type (I). The mapping
\begin{equation*}
    g:\RR\to\RR^4 : s\mapsto (s^3,s^2,s,1)\cdot A
\end{equation*}
gives an arc of $C^\kappa$ containing the point $f$, which has the
parameter $s=0$. The products of $g$ with $G_i$ are functions
{\arraycolsep0.5mm\begin{eqnarray*}
  s&\mapsto& (-a_{11}^2+a_{33})s^3 +
             (-a_{01}a_{11}+a_{23})s^4+(*), \\
  s&\mapsto& (-a_{11}^2+a_{11}a_{33})s^2+
             (a_{01}a_{33}+a_{11}a_{23})s^3+
             (a_{01}a_{23}-2a_{11}a_{02}+a_{11}a_{13})s^4+(*),
\end{eqnarray*}}
where $(*)$ denotes terms of higher order in $s$. The twisted
cubics $C$ and $C^\kappa$ have contact of order $m$ at $f$ if, and
only if, in both functions the coefficients at
$s^0,s^1,\ldots,s^m$ vanish \cite[p.~147]{bol-50}. Now the
assertions follow easily, taking into account that $a_{11}\neq 0$
and $a_{33}\neq 0$.
\par
Similarly, if the matrix $A$ is of type (II) then the functions
{\arraycolsep0.5mm\begin{eqnarray*}
  s&\mapsto& (-a_{11}a_{22}+a_{22})s^3+
             (-a_{01}a_{22}-a_{11}a_{12})s^4+(*), \\
  s&\mapsto& (a_{11}a_{22}-a_{22}^2)s^2+
             (a_{01}a_{22}-2a_{12}a_{22})s^3+
             (-2a_{02}a_{22}+a_{11}a_{13}-a_{12}^2)s^4+(*),
\end{eqnarray*}}
are obtained, whereas for an $A$ of type (III) we get
{\arraycolsep0.5mm\begin{eqnarray*}
  s&\mapsto& (-a_{22}+a_{33})s^3+
             (-a_{12}+a_{23})s^4 + (*), \\
  s&\mapsto& (-a_{22}^2+a_{33})s^2+(-2a_{12}a_{22}+a_{23})s^3
            +(-2a_{02}a_{22}-a_{12}^2+a_{13})s^4 + (*).
\end{eqnarray*}}
As above, the results are immediate.
\end{proof}
Let us now consider $\Phi$ as \emph{plane at infinity}. Then our
projective frame of reference determines an affine coordinate
system in the usual way; a point $\RR(1,x_1,x_2,x_3)\in\PP_3(\RR)$
has affine coordinates $(x_1,x_2,x_3)$. It is our aim to describe
the results of Theorem \ref{thm:kubik} and Theorem \ref{thm:2} in
affine terms. From the affine point of view the twisted cubics $C$
and $C^\kappa$ are \emph{cubic parabolas}, projectively extended
by the point $f=p_3$. So this point of higher order contact is
\emph{outside} the affine space. In what follows an \emph{affine
transformation}\/ is understood to be a collineation fixing the
plane $\Phi$. We restrict ourselves to the description of higher
order contact via regular matrices of type (I). Such a matrix $A$
admits the following factorization:
\begin{equation}\label{eq:zerlegung}
    \Mat4{1 & 0 &0 &0 \\
        0 & 1 &0      & 0\\
        0 & 0      &1 & 0\\
        0 & 0      &0 & \frac{a_{33}}{\raisebox{0.8mm}{$\scriptstyle a_{11}$}}
    }\cdot
    \Mat4{1 & 0 &0 & 0\\
        0 & 1   &0 & \frac{a_{13}}{\raisebox{0.8mm}{$\scriptstyle a_{11}$}}\\
        0 & 0      &1 & \frac{a_{23}}{\raisebox{0.8mm}{$\scriptstyle a_{11}$}}\\
        0 & 0      &0      & 1
    }\cdot
    \Mat4{1 & 0 &0 & 0\\
        0 & a_{11} &0      & 0\\
        0 & 0      &a_{11} & 0\\
        0 & 0      &0      & a_{11}
    }\cdot
    \Mat4{1 & a_{01} &a_{02} & a_{03}\\
        0 & 1  &0      & 0\\
        0 & 0      &1 & 0\\
        0 & 0      &0      & 1
    }
\end{equation}
Conversely, if the entries $a_{ij}$ in (\ref{eq:zerlegung}) are
chosen arbitrarily, except for $a_{11},a_{33}\neq 0$, then a
regular matrix of type (I) is obtained. Formula
(\ref{eq:zerlegung}) corresponds to a decomposition of $\kappa$
into a \emph{perspective affinity}\/ with axis $x_3=0$ in the
direction of $p_3$, a \emph{shear}\/ with an axis through the line
$x_1=x_2=0$ in the direction of $p_3$, a \emph{stretching}\/
fixing the origin $p_0$ with scale factor $a_{11}$, and a
\emph{translation}\/ through the vector $(a_{01}, a_{02} ,
a_{03})$, respectively; this decomposition is uniquely determined.
\par
The matrix $A$ is of type (I.1) if, and only if, $a_{11}=a_{33}$,
i.e., the first matrix in (\ref{eq:zerlegung}) is the unit matrix.
The ultimate and the penultimate matrix in (\ref{eq:zerlegung})
together yield a \emph{dilatation}\/ and every dilatation arises
in this way. Hence, up to dilatations, we obtain all twisted
cubics which have second order contact with $C$ at $f$ by applying
to $C$ all shears with the properties mentioned above. Figure
\ref{abb:2} shows the twisted cubic $C$ and some of its images
under a group $\Sigma$ of shears in the direction of $p_3$ with
the common axis $x_1+x_2=0$. All these twisted cubics are on a
parabolic cylinder $\Psi$ ($x_1^2-x_2=0$) which is invariant under
the group $\Sigma$.
      {\unitlength1cm
\begin{center}
\begin{minipage}[t]{3.81cm}
         \begin{picture}(3.81,5)
         \put(0.0,0.0){\includegraphics[height=5cm]{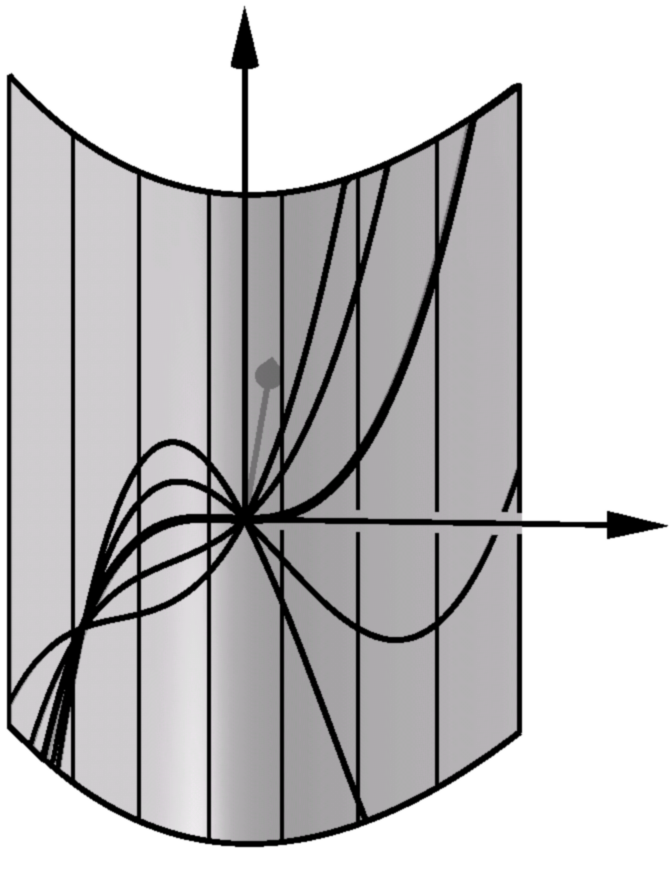}}
         \put(2.5,4.5){\small$C$}
         \put(3.5,2.2){\small$x_1$}
         \put(1.5,4.7){\small$x_3$}
         \put(3.1,3.6){\small$\Psi$}
         \end{picture}
         \refstepcounter{abbildung}\label{abb:2}
         \centerline{\small\hspace{-6mm}Figure \ref{abb:2}}
\end{minipage}
\hspace{.7cm}
\begin{minipage}[t]{4.09cm}
         \begin{picture}(4.09,5)
         \put(0.0 ,0.0){\includegraphics[height=5cm]{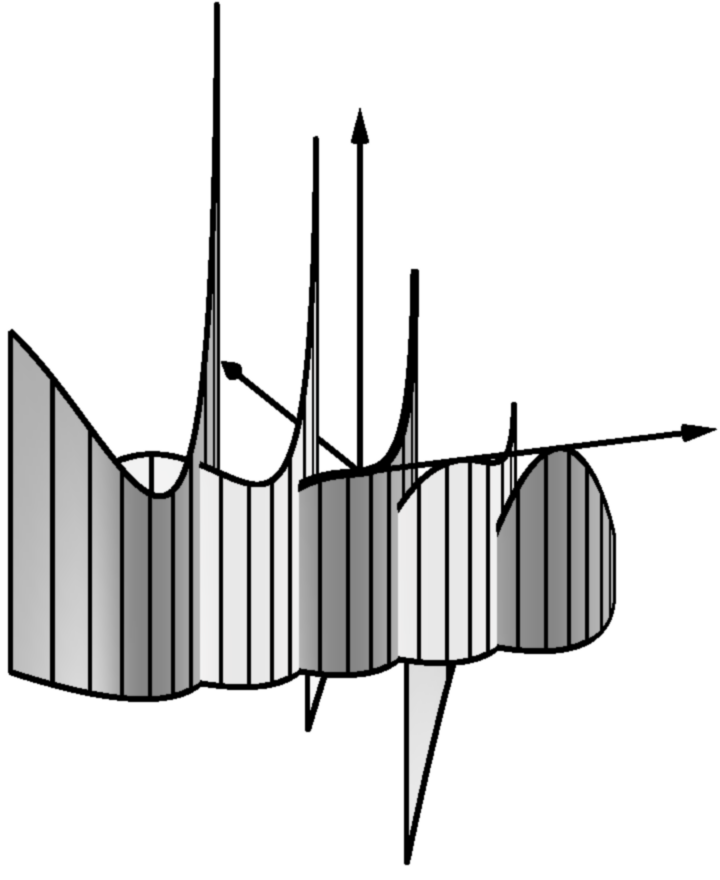}}
         \put(2.15,3.55){\small$C$}
         \put(3.75,2.7){\small$x_1$}
         \put(1.3,3.0){\small$x_2$}
         \put(2.1,4.4){\small$x_3$}
         \put(2.4,2.7){\small$\Psi$}
         \end{picture}
         \refstepcounter{abbildung}\label{abb:3}
         \centerline{\small Figure \ref{abb:3}}
\end{minipage}
\hspace{.7cm}
\begin{minipage}[t]{4.38cm}
         \begin{picture}(4.38,5.0)
         \put(0.0 ,0.0){\includegraphics[height=5cm]{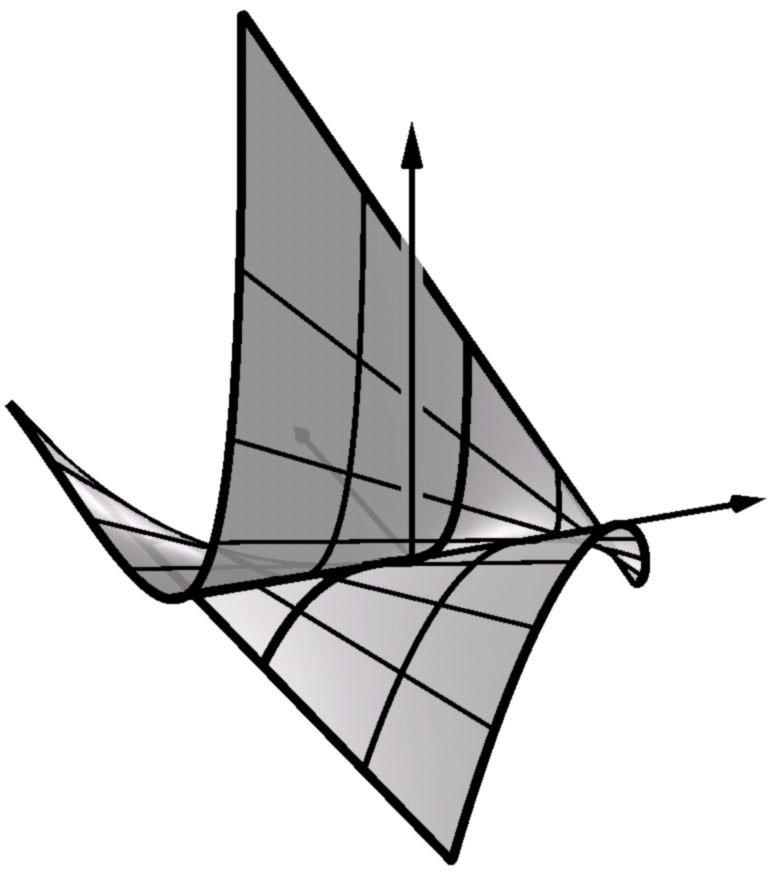}}
         \put(2.6,3.2){\small$C$}
         \put(4.0,2.3){\small$x_1$}
         \put(2.4,4.3){\small$x_3$}
         \end{picture}
         \refstepcounter{abbildung}\label{abb:4}
         \centerline{\small Figure \ref{abb:4}}
\end{minipage}\\
\end{center}
         }%
\par
The matrix $A$ is of type (I.1.2) if, and only if, it can be
written as
\begin{equation}\label{eq:zerlegung12}
    \Mat4{1 & 0 &0 & 0\\
        0 & 1   &0 & a_{13}\\
        0 & 0      &1 & 0\\
        0 & 0      &0      & 1
    }\cdot
    \Mat4{1 & 0 &0 & 0\\
        0 & 1   &0 & 0\\
        0 & 0      &1 & -a_{01}\\
        0 & 0      &0      & 1
    }\cdot
    \Mat4{1 & a_{01} &0 & 0\\
        0 &   1 &0      & 0\\
        0 & 0      &1 & 0\\
        0 & 0      &0      & 1
    }\cdot
    \Mat4{1 & 0 &a_{02} & a_{03}\\
        0 & 1   &0      & 0\\
        0 & 0   &1 & 0\\
        0 & 0   &0      & 1
    }.
\end{equation}
As before, this factorization is unique and the coefficients can
be chosen freely. The first (second) matrix gives a shear with
axis $x_1=0$ ($x_2=0$) in the direction of $p_3$, whereas the
remaining matrices yield a translation in the direction of $p_1$
and a translation parallel to the plane $x_1=0$. However, the
second and the third matrix are linked via the common parameter
$a_{01}$. As $a_{01}$ varies in $\RR$, their products comprise a
one-parameter group $\Gamma_1$ of affine transformations. (See
\cite[I, p.~130]{brau-66-67}, III 3, ``\emph{Nichtisotrope
Cliffordschiebungen}'': All points of the line $x_0=x_2=0$ are
invariant under $\Gamma_1$. All other point orbits are lines of a
parabolic linear congruence with axis $x_0=x_2=0$.) Hence, up to
translations parallel to the plane $x_1=0$, we obtain all twisted
cubics which have third order contact with $C$ at $f$ by applying
to $C$ all shears with axis $x_1=0$ and then all transformations
of the group $\Gamma_1$. In Figure \ref{abb:3} the twisted cubic
$C$ and some of its images under affinities of $\Gamma_1$ are
displayed. These curves lie on parabolic cylinders which are
translates of $\Psi$. Figure \ref{abb:4} shows the ruled surface
which arises by applying $\Gamma_1$ to the curve $C$. The
illustrated lines are point orbits with respect to $\Gamma_1$. In
particular, the $x_1$-axis of the coordinate system is the orbit
of the origin; this line is an edge of regression of the surface.
\par
The matrix $A$ is of type (I.1.2.3) if, and only if, it can be
written as
\begin{equation}\label{eq:zerlegung123}
     \Mat4{1 & 0 &0 & 0\\
        0 & 1   &0 &2a_{02}\\
        0 & 0      &1 & 0\\
        0 & 0      &0      & 1
    }\cdot
        \Mat4{1 & 0 &a_{02} & 0\\
        0 & 1   &0 & 0\\
        0 & 0      &1 & 0\\
        0 & 0      &0  & 1
    }\cdot
    \Mat4{1 & 0 &0 & a_{03}\\
        0 & 1   &0      & 0\\
        0 & 0   &1 & 0\\
        0 & 0   &0      & 1
    }.
\end{equation}
Again, this decomposition is unique and the coefficients can be
chosen arbitrarily. The products of the first and the second
matrix in (\ref{eq:zerlegung123}) comprise a one-parameter
subgroup $\Gamma_2$; cf.\ the remarks above. Hence, up to
translations parallel to the line $x_1=x_2=0$, we obtain all
twisted cubics which have fourth order contact with $C$ at $f$ as
the orbit of $C$ under $\Gamma_2$. Figure \ref{abb:5} illustrates
the twisted cubic $C$ and the cylinder $\Psi$, together with some
of their images under affinities of $\Gamma_2$. Figure \ref{abb:6}
shows the ruled surface which is generated by applying $\Gamma_2$
to the curve $C$. This surface is a proper subset (only the points
of $F\setminus\{f\}$ are missing) of the (ruled) \emph{Cayley
surface} with equation $2x_0x_1x_2-x_1^3=x_0^2x_3$.%
{\unitlength1cm
\begin{center}
\begin{minipage}[t]{3.64cm}
         \begin{picture}(3.64,4.6)
         \put(0.0 ,0.0){\includegraphics[height=5cm]{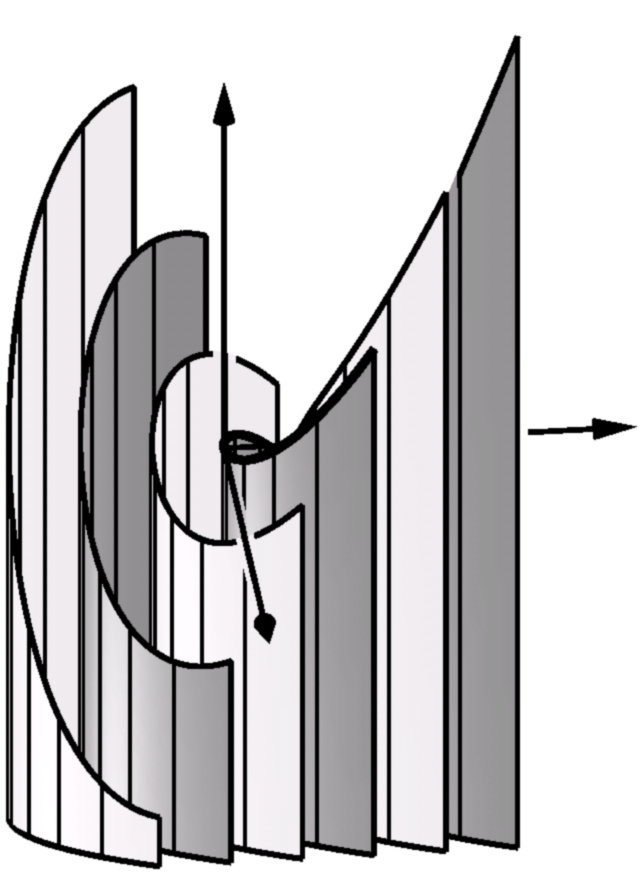}}
         \put(1.8,2.35){\small$C$}
         \put(1.8,0.4){\small$\Psi$}
         \put(3.3,2.75){\small$x_2$}
         \put(1.41,1.05){\small$x$}
         \put(1.58,1.05){\small${}_1$}
         \put(1.4,4.4){\small$x_3$}
         \end{picture}
         \refstepcounter{abbildung}\label{abb:5}
         \centerline{\small Figure \ref{abb:5}}
\end{minipage}
\hspace{2cm}
\begin{minipage}[t]{3.78cm}
         \begin{picture}(3.78,4.5)
         \put(0.0 ,0.0){\includegraphics[height=5cm]{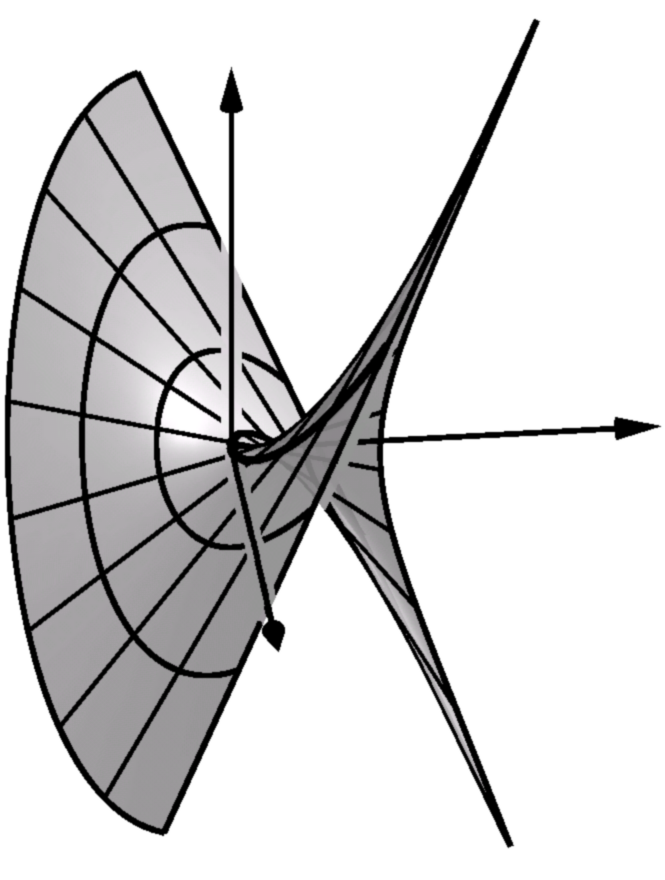}}
         \put(3.35,2.75){\small$x_2$}
         \put(1.45,1.05){\small$x_1$}
         \put(1.45,4.45){\small$x_3$}
         \end{picture}
         \refstepcounter{abbildung}\label{abb:6}
         \centerline{\small Figure \ref{abb:6}}
\end{minipage}
\end{center}}

\section{The Laguerre Geometry
$\Sigma(\RR,\LL)$}\label{sect:laguerre}

Let $\RR[X]$ be the polynomial ring over the reals. The factor
ring $\RR[X]/\langle X^3\rangle=:\LL$  is a $3$-dimensional real
commutative local algebra with an $\RR$-basis $1_\LL,\eps,\eps^2$,
where $\eps:=X+\langle X^3\rangle$. Its non-invertible elements
form the only maximal ideal $N:=\RR\eps+\RR\eps^2$. We consider
$\RR$ as a subring of $\LL$ by identifying $x\in\RR$ with $x\cdot
1_\LL\in \LL$. (Our ring $\LL$ is the ring $\mathfrak L_4$ in
\cite[p.~306]{benz-73}.) Let us recall the definition of the
\emph{projective line over}\/ $\LL$, in symbols $\PP(\LL)$: We
consider the free left $\LL$-module $\LL^2$. A cyclic submodule
$\LL(u,v)\subset \LL^2 $ is a point of $\PP(\LL)$ if, and only if,
$u$ or $v$ is a unit in $\LL$. Two such pairs $(u,v)$ and
$(u',v')$ in $\LL^2$ determine the same point precisely when they
are proportional by a unit in $\LL$. (Cf.\ \cite[p.~785]{herz-95}
for a definition of the projective line over an arbitrary ring
with a unit element.). We embed the real projective line
$\PP(\RR)$ in $\PP(\LL)$ by $\RR(x,y)\mapsto \LL(x,y)$. The point
set of the \emph{chain geometry}\/ $\Sigma(\RR,\LL)$ is the
projective line over $\LL$, the \emph{chains}\/ are the images of
$\PP(\RR)\subset\PP(\LL)$ under the natural right action of
$\GL_2(\LL)$ on $\LL^2$; cf.\ \cite[p.~790]{herz-95}. Since $\LL$
is a local ring, our chain geometry is a \emph{Laguerre
geometry}\/ \cite[p.~793]{herz-95}. If two distinct points of
$\PP(\LL)$ can be joined by a chain then they are said to be
\emph{distant}.  Non-distant points are also said to be
\emph{parallel}\/ ($\parallel$).  Letting $p=\LL(a,b)$ and $q=
\LL(c,d)$ gives
\begin{equation}\label{eq:parallel}
  p\,\parallel q\,\Leftrightarrow\, \det\Mat2{a&b\\c&d}\in N.
\end{equation}
This parallelism is an equivalence relation. There is a unique
chain through any three mutually distant points.
\par
We fix the point $\LL(1,0)=:\infty\in\PP(\LL)$. Then the point set
of $\PP(\LL)$ can be split into two classes: A \emph{proper
point}\/ has the form $\LL(z,1)$, and we identify such a point
with the element $z\in \LL$. The proper points are precisely the
points which are distant (non-parallel) to $\infty$. Every other
point of $\PP(\LL)$ has the form $\LL(1,z)$ with $z\in N$. Such
points are said to be \emph{improper}. Hence we can regard
$\PP(\LL)$ as the real affine $3$-space on $\LL$ together with an
extra ``improper plane'' which is just a copy of the maximal ideal
$N$.
\par
The algebra $\LL$ has two distinguished ideals, namely the maximal
ideal $N$ and its annihilator $\{z\in \LL\mid zN=0\}=\RR\eps^2$.
Accordingly, there are three types of lines: A line $\RR
u+v\subset \LL$, where $u\in \LL\setminus\{0\}$, $v\in \LL$ is
called \emph{singular}\/ if $u\in N$, and \emph{regular}\/
otherwise. A singular line of the form $\RR\eps^2+v$ is said to be
\emph{vertical}. We say that a plane is \emph{regular}\/ provided
that it contains at least one regular line. A \emph{singular}\/
plane is just a non-regular plane. By (\ref{eq:parallel}), the
singular planes are the classes of proper parallel points.
\par
For each subset $\cS\subset\PP(\LL)$ let $\cS^\circ$ be its
\emph{proper part}, i.e. the set of all its proper points. The
following is taken from \textsc{H.-J.~Samaga}
\cite[Satz~4]{sama-76}; cf.\ also \cite{benz+s+s-81}: A subset
$\cC$ of $\PP(\LL)$ is a chain of $\Sigma(\RR,\LL)$ precisely when
one of the following conditions holds:
\begin{equation}\label{eq:gerade}
  \cC=\{t+(a_{02}+a_{12} t)\eps+(a_{03}+a_{13} t)\eps^2\mid
  t\in\RR\}\cup\{\infty\},
\end{equation}
whence $\cC^\circ$ is an affine line;
\begin{equation}\label{eq:parabel}
  \cC=\{t+(a_{02}+a_{12} t)\eps+(a_{03}+a_{13} t+a_{33} t^2)\eps^2\mid t\in\RR, a_{33}\neq 0\}
  \cup \{ \LL(1,-a_{33}\eps^2) \},
\end{equation}
whence $\cC^\circ$ is a parabola;
{\arraycolsep0.5mm\begin{eqnarray}
  \cC&=&\{t+(a_{02}+a_{12} t+a_{22} t^2)\eps+(a_{03}+a_{13} t+a_{23} t^2+a_{33} t^3)\eps^2
          \mid t\in\RR, a_{33}=a_{22}^2\neq  0\}\nonumber\\
  &&{} \cup\{ \LL(1,-a_{22}\eps+(-a_{23}+2a_{12}a_{22})\eps^2)\},\label{eq:kub_parabel}
\end{eqnarray}}
whence $\cC^\circ$ is a cubic parabola. In either case the
$a_{ij}$'s are real constants subject to the conditions stated
above. Obviously, the lines given by (\ref{eq:gerade}) are
precisely the regular ones. So, all regular lines are
\emph{representatives} for the point $\infty$. We say that a
(cubic) parabola in $\LL$ is \emph{admissible}\/ if it is the
proper part of a chain. By (\ref{eq:parabel}), a parabola is
admissible if, and only if, its diameters are vertical lines and
its plane is regular. Each admissible parabola is a representative
of a unique improper point. We describe admissible parabolas which
determine the same improper point:

\begin{thm}\label{thm:parabeln}
Let $\cC^\circ$ and $\welle{\cC}^\circ$ be admissible parabolas
of\/ $\LL$. Then the chains $\cC$ and $\welle{\cC}$ have the same
improper point if, and only if, the parallel projection of
$\welle{\cC}^\circ$ onto the plane of $\cC^\circ$, in the
direction of an arbitrary non-vertical singular line, is a
translate of $\cC^\circ$.
\end{thm}
\begin{proof}
Let $\cC$ and $\welle{\cC}$ be given according to
(\ref{eq:parabel}) with coefficients $a_{ij}$ and $\welle a_{ij}$,
respectively. The parallel projection of $\welle \cC^\circ$ onto
the plane of $\cC^\circ$ is a parabola
\begin{equation*}
 \{t+(a_{02}+a_{12} t)\eps
    +(a_{03}^*+a_{13}^* t+a_{33}^* t^2)\eps^2
    \mid t\in\RR\}
 \mbox{ with } a_{33}^*=\welle{a}_{33}.
\end{equation*}
An easy calculation shows that the projected parabola is a
translate of $\cC^\circ$ if, and only if, $\welle a_{33}=a_{33}$.
By (\ref{eq:parabel}), this is necessary and sufficient for $\cC$
and $\welle \cC$ to have the same improper point.
\end{proof}

Let us consider the \emph{projective closure}\/ $\PP_3(\RR)$ of
the affine space on $\LL$, where we do not distinguish between
$\RR(1,x_1,x_2,x_3)\in\PP_3(\RR)$ and $x_1+x_2\eps+x_3\eps^2\in
\LL$. Since we are going to work with \emph{two different
extensions}\/ of the affine space on $\LL$, we reserve the phrases
``at infinity'' and ``improper'' for the projective closure and
for the chain-geometric closure, respectively. If $\cC$ is a chain
of $\Sigma(\RR,\LL)$ then $\cC^+\subset\PP_3(\RR)$ denotes that
unique projective line or conic or twisted cubic which contains
$\cC^\circ$. We denote by $f$, $F$, and $\Phi$ the point at
infinity of the vertical line $\RR\eps$, the line at infinity of
the singular plane $N$, and the plane at infinity, respectively.
\par
Let $\cC$ be a chain. If $\cC^\circ$ is a line then
$\cC^+\not\subset\Phi$ is a projective line with a point at
infinity not on $F$ and vice versa. If $\cC^\circ$ is a parabola
then $\cC^+\not\subset\Phi$ is a conic through $f$ touching a line
at infinity other than $F$. As before, all such conics arise from
chains. We note that when $\cC^\circ$ and $\welle{\cC}^\circ$ are
parabolas in the same plane then the existence of a translation
taking $\cC^\circ$ to $\welle{\cC}^\circ$ just means that the
projective conics $\cC^+$ and $\welle{\cC}^+$ have contact of
second order at the point $f$. See, e.g., \cite{pauk-79}. But,
since admissible parabolas in different planes may represent the
same improper point, we cannot always describe improper points in
terms of conics with second order contact at infinity. Now we turn
to the case when $\cC^\circ$ is a cubic parabola:
\begin{thm}\label{thm:3}
The cubic parabola
\begin{equation}\label{eq:standard_kub_parabel}
  \{t+t^2\eps+ t^3\eps^2\mid t\in\RR\}
\end{equation}
is admissible. A cubic parabola of\/ $\LL$ is admissible if, and
only if, its projective extension and the projective extension of
(\ref{eq:standard_kub_parabel}) have contact of second order at
the point $f=\RR(0,0,0,1)$.
\end{thm}
\begin{proof}
By (\ref{eq:kub_parabel}), there is a unique chain $\cD$, say,
such that $\cD^\circ$ coincides with the cubic parabola
(\ref{eq:standard_kub_parabel}). Its projective extension $\cD^+$
is given by (\ref{eq:kubik-homogen}), whence it is a twisted cubic
through $f$, with tangent line $F$, and osculating plane $\Phi$.
Now we apply those collineations of $\PP_3(\RR)$ which are given
by regular matrices of type (III.1). So we get all the twisted
cubics which have second order contact with $\cD^+$ at $f$ and, by
(\ref{eq:kub_parabel}), these are precisely the projectively
extended admissible cubic parabolas.
\end{proof}
\begin{thm}\label{thm:4}
Let $\cC^\circ$ and $\welle{\cC}^\circ$ be admissible cubic
parabolas of\/ $\LL$. Then the chains $\cC$ and $\welle{\cC}$ have
the same improper point if, and only if, the extended curves
$\cC^+$ and $\welle{\cC}^+$ have contact of third order at
$f=\RR(0,0,0,1)$.
\end{thm}
\begin{proof} (a) First, we consider that chain $\cD$ which yields the cubic
parabola (\ref{eq:standard_kub_parabel}). The improper point of
$\cD$ is $\LL(1,-\eps)$. Now we apply those collineations of
$\PP_3(\RR)$ which are given by regular matrices of type
(III.1.2). This gives, by Theorem \ref{thm:2}, precisely those
twisted cubics which have third order contact with $\cD^+$ at $f$
and, by (\ref{eq:kub_parabel}), we get all the projectively
extended cubic parabolas that arise from the chains through
$\LL(1,-\eps)$. Since contact of third order is a transitive
notion, the assertion follows for all chains $\cC$ through
$\LL(1,-\eps)$.
\par
(b) Next, let $\cC$ be any chain whose proper part is a cubic
parabola, so that  its improper point can be written as
$\LL(1,-a\eps-b\eps^2)$, where $a,b\in\RR$ and $a\neq 0$. The
matrix
\begin{equation*}
    \alpha:=\left(\begin{array}{rr}
    a & 0 \\ -\frac{b}{\raisebox{0.8mm}{$\scriptstyle a$}} & 1
    \end{array}\right)
    \in\GL_2(\RR)\subset\GL_2(\LL)
\end{equation*}
induces a projectivity of $\PP(\LL)$ taking the improper point
$\LL(1,-a\eps-b\eps^2)$ to
\begin{equation*}
 \LL\Big(\Big(\frac 1 a-\frac b{a^2}\,\eps\Big)\Big(a+b\eps+\frac
{b^2}{a}\,\eps^2,-a\eps-b\eps^2\Big)\Big)
  = \LL(1,-\eps).
\end{equation*}
The action of $\alpha$ on the proper points is the affine
transformation $\LL\to \LL: z\mapsto za-\frac{b}{a}$ which in turn
can be extended to a collineation of $\PP_3(\RR)$. Since contact
of any order is preserved under collineations, we can apply the
results from (a) in order to complete the proof.
\end{proof}
From the affine point of view, the previous results are not
satisfying, because they are formulated in projective terms.
However, in Section \ref{sect:kubiken} we have explained how one
can ``see'' contact of higher order at $f$ via an affine
transformation taking $\welle \cC^\circ$ to $\cC^\circ$. Another
basic topic is to characterize chains $\cC$ and $\welle \cC$ which
\emph{touch}\/ at a common improper point. If $\cC^\circ$ is an
affine line then this means, by definition, that $\cC^\circ$ and
$\welle \cC$ are parallel lines. If $\cC^\circ$ is a parabola then
a characterization as in Theorem \ref{thm:2} can be given, but now
the parallel projection of $\welle \cC^\circ$ has to arise from
$\cC^\circ$ under a translation in the direction of $\eps^2$.
(This means contact of third order at $f$.) Likewise, Theorem
\ref{thm:4} can be modified as to describe touching chains, by
replacing ``third order contact'' with ``fourth order contact''.
The proofs are left to the reader. The affine space on $\LL$ is
closely related with the \emph{flag space}\/ (\emph{two-fold
isotropic space}\/), as the triple $(f,F,\Phi)$ can be considered
as its \emph{absolute flag}. Cf., among others, the papers
\cite{brau-66-67} by \textsc{H.~Brauner}. Due to lack of space we
have to refrain from presenting here the interesting connections
between these two geometries.


{\small
  Hans Havlicek and Klaus List\\
  Institut f\"ur Geometrie\\
  Technische Universit\"at\\
  Wiedner Hauptstra{\ss}e 8--10\\
  A-1040 Wien\\
  Austria
}
\end{document}